\documentclass[10pt, a4paper, leqno]{amsart}
\usepackage[utf8]{inputenc} 
\usepackage[english]{babel} 
\usepackage[usenames, dvipsnames]{xcolor}
\usepackage[a4paper, tmargin=1.0in,bmargin=1.0in, lmargin=1in, rmargin=1in]{geometry} 
\usepackage{bm} 
\usepackage{amsmath} 
\usepackage{amsthm}
\usepackage{amssymb}
\usepackage{dsfont}
\usepackage{cases} 
\usepackage{tikz} 
\usetikzlibrary{arrows.meta, decorations.markings, calc}
\usepackage{stmaryrd} 
\usepackage{multirow} 
\usepackage{lipsum} 
\usepackage{wrapfig} 

\usepackage{breakcites} 
\usepackage{mathtools}   
\usepackage{cases}
\usepackage{float}
\usepackage{esint}

\usepackage{todonotes}

\usepackage{enumerate}
\usepackage[shortlabels]{enumitem}

\usepackage{pgfplots}
\pgfplotsset{compat=1.15}

\usepackage[colorlinks=true,linkcolor=BrickRed,citecolor=RoyalBlue]{hyperref} 
\usepackage{cleveref} 

\usepackage{mathabx}

\theoremstyle{plain}
\newtheorem{proposition}{Proposition}%
\newtheorem{lemma}{Lemma}%
\newtheorem*{theorem*}{Theorem}
\theoremstyle{definition}
\theoremstyle{remark}
\newtheorem{remark}{Remark}%

\newcommand{\xLabel}{x}
\newcommand{\yLabel}{y}
\newcommand{\momentum}{q}
\newcommand{\velocity}{u}
\newcommand{\momentumX}{q_{\xLabel}}
\newcommand{\velocityX}{u_{\xLabel}}
\newcommand{\momentumY}{q_{\yLabel}}
\newcommand{\velocityY}{u_{\yLabel}}
\newcommand{\density}{\rho}
\newcommand{\distributionFunction}{f}
\newcommand{\timeVariable}{t}

\newcommand{\vectorial}[1]{\bm{#1}}
\newcommand{\zeroVelocity}{\circ}
\newcommand{\xPlusVelocity}{\smalltriangleright}
\newcommand{\xMinusVelocity}{\smalltriangleleft}
\newcommand{\yPlusVelocity}{\smalltriangleup}
\newcommand{\yMinusVelocity}{\smalltriangledown}
\newcommand{\collided}{\star}
\newcommand{\timeStep}{\Delta t}
\newcommand{\spaceStep}{\Delta x}
\newcommand{\naturals}{\mathbb{N}}
\newcommand{\relatives}{\mathbb{Z}}
\newcommand{\indexVelocity}{\ell}
\newcommand{\pressureLagrange}{\Phi}
\newcommand{\viscosity}{\nu}
\newcommand{\latticeBoltzmannScheme}[3]{$\textrm{D}_{#1}\textrm{Q}_{#2}^{#3}$}
\newcommand{\diagonalMatrix}{\mathbf{diag}}
\newcommand{\relaxationParameter}{\omega}
\newcommand{\atEquilibrium}{\textrm{eq}}
\newcommand{\linearEquilibrium}{\alpha}
\newcommand{\parabolicScaling}{\mu}
\newcommand{\pressure}{P}
\newcommand{\antiSymmetricMoment}{a}
\newcommand{\symmetricMoment}{s}
\newcommand{\BigO}[1]{\mathcal{O}(#1)}
\newcommand{\densityIncompressible}{\overline{\density}}
\newcommand{\referenceState}[1]{\overline{#1}}
\newcommand{\fourierTransformed}[1]{\hat{#1}}
\newcommand{\frequency}{\xi}

\newcommand{\reals}{\mathbb{R}}
\newcommand{\eps}{\varepsilon}
\newcommand{\Maxwellian}{\vectorial{\mathcal{M}}}
\newcommand{\strong}[1]{\emph{#1}}

\title{A vectorial lattice Boltzmann scheme for the incompressible Navier-Stokes equations}

\author{Denise Aregba-Driollet}
\address{Université de Bordeaux, CNRS, Bordeaux INP, IMB, UMR 5251, 33400 Talence, France}
\email{aregba@math.u-bordeaux.fr}

\author{Thomas Bellotti}
\address{Université Paris-Saclay, CNRS, CentraleSupélec, Laboratoire EM2C \& Fédération de Mathématiques de CentraleSupélec, 91190, Gif-sur-Yvette, France}
\email{thomas.bellotti@centralesupelec.fr}

\author{Roberto Natalini}
\address{Istituto per le Applicazioni del Calcolo, Consiglio Nazionale delle Ricerche, Rome, Italy}
\email{roberto.natalini@cnr.it}

\author{Tommaso Tenna}
\address{Laboratoire J. A. Dieudonné, Université Côte d’Azur, CNRS, F-06108 Nice, France \& Dipartimento di Matematica ``Guido Castelnuovo'', Sapienza Università di Roma, 00185 Rome, Italy}
\email{tommaso.tenna@uniroma1.it}


\begin{document}

\maketitle

\begin{abstract}
    We introduce a second-order accurate vectorial lattice Boltzmann scheme for the incompressible Navier-Stokes system, inspired by a discrete-velocity kinetic approximation proposed by Carfora and Natalini [ESAIM: M2AN, 42(1), 93-112, 2008].
    Advantages and drawbacks compared to relaxation schemes are investigated by providing spectral analyses in the linearized case, and numerical validations on the genuinely non-linear problem.
\end{abstract}


\section{Introduction}

In this work we propose a \strong{lattice Boltzmann method} (often abridged with LBM) for the numerical approximation of the \strong{incompressible Navier-Stokes} system 
\begin{equation}\label{eq:navierStokesSystem}
    \nabla\cdot \vectorial{\velocity} = 0, 
    \qquad
    \partial_{\timeVariable}\vectorial{\velocity}
    +
    \nabla\cdot (\vectorial{\velocity}\otimes\vectorial{\velocity})
    +
    \nabla \pressureLagrange 
    -\viscosity\Delta \vectorial{\velocity}
    =
    \vectorial{0},
\end{equation}
with $\vectorial{\velocity}: \reals_+ \times \reals^d \to \reals^d$ representing the velocity field, endowed with an initial condition.
Here, $\pressureLagrange$ is a Lagrange multiplier ensuring incompressibility and $\viscosity>0$ a viscosity.
For the sake of illustration, we essentially consider the two-dimensional setting with $d = 2$.
However, the extension to $d = 3$ is straightforward thanks to the Cartesian character of the numerical scheme.

Several LBMs have been proposed for the approximation of \eqref{eq:navierStokesSystem}: seminal contributions include \cite{Frisch1986, PhysRevLett.61.2332, chen1992recovery, qian1992lattice, heluo1997, succi2001}, whose schemes---based on either on hexagonal or Cartesian grids---recover the incompressible Navier-Stokes equation in the low-Mach limit. 
Along this vein, a discrete velocity model based on a lattice Boltzmann method for  \eqref{eq:navierStokesSystem} was subsequently provided by Junk and Klar \cite{junkklar2000}, who established a rigorous asymptotic analysis of the so-derived numerical scheme and its macroscopic limit, see also \cite{junk2009convergence}. For the sake of completeness, we also mention \cite{guo2000}, in which a lattice Boltzmann BGK models (LBGK) without compressible effect is designed for simulating incompressible flows. For a review on the topic, one can refer to \cite{succi2001, kruger2017lattice}. 

A common feature of these approaches is that a \strong{scalar} distribution function is employed, and thus the (almost constant) density is recovered as zero-order moment of the distribution function, whereas the \strong{momentum} is a first-order moment in the discrete velocities.
A well-known feature of these numerical schemes is their lack of Galilean invariance, which manifests via the appearance of cubic terms in the flow velocity in the stress tensor.
Moreover, this approach has limited flexibility when additional equations are present in the target system, e.g. \cite{kataoka2004lattice}, since a large number of discrete velocities is needed in this context.

Another approach \cite{natalini1998discrete, aregba2000discrete}, which we indeed follow, is to consider a \strong{vectorial} distribution function, which has as many components as the number of conserved quantities, which are all zero-order moments in the velocities. 
This approach has driven a significant number of recent works, such as \cite{bellotti2025fourth, wissocq2025positive} with LBMs for hyperbolic systems.
Moreover, a very recent paper \cite{dimarco2025asymptotic} followed this path and proposed a high-order approximation of a discrete velocity method inspired by the LBM for the incompressible Navier-Stokes equations.

The numerical scheme that we introduce is inspired by the discrete BGK-type approximation proposed by Carfora and Natalini in \cite{carfora2008discrete}, later analyzed in detail by \cite{bouchut2018}. 
The kinetic model reads as follows
\begin{equation}
\label{eq:discreteBGKmodel_NS}
    \partial_{\timeVariable} \vectorial{\distributionFunction}_{\ell} + \frac{\vectorial{\lambda}_{\ell}}{\eps} \cdot\nabla_{\vectorial{x}} \vectorial{\distributionFunction}_{\ell} = \frac{1}{\tau\,\eps^2} \left(\Maxwellian_{\ell}(\density, \eps \density \vectorial{\velocity}) - \vectorial{\distributionFunction}_{\ell} \right), \qquad {\ell}=1,\dots,L,
\end{equation}
where $L \geq d+1$ is the number of discrete velocities, $\vectorial{\distributionFunction}_{\ell}$ and $\Maxwellian_{\ell}$ take values in $\reals^{d+1}$, $\eps>0$ is the relaxation parameter and $\vectorial{\lambda}_{\ell} = (\lambda_{{\ell}1},\dots,\lambda_{{\ell}d})$ is the $\reals^d$-vector of discrete velocities. 
Here, $\density$ is a density field, as the approach introduces some \strong{artificial compressibility}, and shall be constant at leading-order.
The functions $\Maxwellian_{\ell}$ are the Maxwellian (or equilibrium) functions, assumed to be Lipschitz continuous. Finally we set 
\begin{equation}
    \begin{pmatrix}
            \density\\
            \eps\,\momentum_1\\
            \vdots\\
            \eps\,\momentum_d
        \end{pmatrix} 
        (\timeVariable,\vectorial{x})
        =
         \begin{pmatrix}
            \density\\
            \eps\,\density\velocity_1\\
            \vdots\\
            \eps\,\density\velocity_d
        \end{pmatrix} 
        (\timeVariable,\vectorial{x})
        = \sum_{{\ell}=1}^L \vectorial{\distributionFunction}_{\ell}(\timeVariable,\vectorial{x}).
\end{equation}
Provided that the discrete velocities $\vectorial{\lambda}_{\ell}$ and the Maxwellian functions $\Maxwellian_{\ell}$ satisfy appropriate consistency conditions as in \cite{carfora2008discrete}, the singular perturbation system formally admits \eqref{eq:navierStokesSystem} as its hydrodynamic limit. In particular, the approximation of this system described in \cite{carfora2008discrete}---indeed a \strong{relaxation scheme}---benefits from the stability condition derived from a discrete-velocity analogue of the Boltzmann H-theorem, as shown in \cite{bouchut1999}. 
The vectorial lattice Boltzmann method can be designed in analogy with the discrete BGK model \eqref{eq:discreteBGKmodel_NS}. 

Before going on, note the following.
In a context where the target equation is a system of hyperbolic conservation law \cite{aregba2000discrete}, relaxation away from the equilibrium in lattice Boltzmann schemes is a way of reducing numerical diffusion \cite{graille2014approximation}, which is a by-product of the discretization.
In this case, relaxation parameters are dictated by the need for reducing numerical diffusion while keeping stability.
When the target problem (e.g., \eqref{eq:navierStokesSystem}) contains dissipation as part of the model, relaxation away from the equilibrium participates---in conjunction with the Maxwellians---to providing the dissipation structure. 
For this reason, one of the questions to be elucidated in this work is whether avoiding a relaxation scheme \cite{carfora2008discrete} is a good choice.

The paper is structured as follows.
In \Cref{sec:numericalScheme} we introduce the numerical scheme, whose second-order consistency with \eqref{eq:navierStokesSystem} is analyzed in \Cref{sec:consistency}.
Then, in \Cref{sec:entropy}, we establish entropy stability in the under-relaxation case.
\Cref{sec:spectra} is devoted to a study of several features of the scheme by investigating spectra of its linearization.
Numerical validations both with periodic and non-trivial boundary conditions are provided in \Cref{sec:numericalExp}.
We finally draw general conclusions and perspectives on our work in \Cref{sec:conclusions}.

\section{Numerical scheme}\label{sec:numericalScheme}

We consider a vectorial \latticeBoltzmannScheme{2}{5}{3} scheme, where the discrete velocities are the same regardless of the moment at hand.
However, the discussion can be easily adapted to different discrete velocities, see \cite{dubois2014simulation}.
Still, we allow for \strong{different relaxation parameters} according to the considered equation.
Let $\timeVariable\in\timeStep\naturals$ and $\xLabel, \yLabel\in\spaceStep(\relatives+\frac{1}{2})$ with $\timeStep>0$ and $\spaceStep>0$.\footnote{The fact of considering grid-points indexed on $\relatives+\frac{1}{2}$ is not essential here where boundaries are not considered, but allows to easily write boundary conditions where walls are half-way between two grid-points.}
It is quite crucial that we consider the following \strong{parabolic (or diffusive) scaling} 
\begin{equation}\label{eq:parabolicScaling}
    \frac{\spaceStep^2}{\timeStep} = \parabolicScaling>0 
    \qquad
    \text{fixed,}
\end{equation}
between time and space discretization.
The algorithm proceeds as follows.
\begin{itemize}
    \item \strong{Relaxation} phase.
    Define
    \begin{equation*}
        \begin{pmatrix}
            \density\\
            \spaceStep\,\momentumX\\
            \spaceStep\,\momentumY
        \end{pmatrix}
        (\timeVariable, \xLabel, \yLabel)
        =
        \sum_{\indexVelocity\in \{\zeroVelocity, \xPlusVelocity, \yPlusVelocity, \xMinusVelocity, \yMinusVelocity\}}
        \vectorial{\distributionFunction}_{\indexVelocity}(\timeVariable, \xLabel, \yLabel),
    \end{equation*}
    and relax following 
    \begin{multline}\label{eq:LBM_scheme}
        \vectorial{\distributionFunction}_{\indexVelocity}^{\collided}(\timeVariable, \xLabel, \yLabel)
        =
        \diagonalMatrix(1-\relaxationParameter_{\density}, 1-\relaxationParameter_{\momentumX}, 1-\relaxationParameter_{\momentumY})
        \vectorial{\distributionFunction}_{\indexVelocity}(\timeVariable, \xLabel, \yLabel)\\
        +
        \diagonalMatrix(\relaxationParameter_{\density}, \relaxationParameter_{\momentumX}, \relaxationParameter_{\momentumY})
        \vectorial{\distributionFunction}^{\atEquilibrium}_{\indexVelocity}
        (\density(\timeVariable, \xLabel, \yLabel), \momentumX(\timeVariable, \xLabel, \yLabel), \momentumY(\timeVariable, \xLabel, \yLabel))
    \end{multline}
    for every $\indexVelocity\in \{\zeroVelocity, \xPlusVelocity, \yPlusVelocity, \xMinusVelocity, \yMinusVelocity\}$.
    The relaxation parameters are taken as $\relaxationParameter_{\density}, \relaxationParameter_{\momentumX}, \relaxationParameter_{\momentumY}\in (0, 2]$.
    Note that taking them equal to one corresponds to a \strong{relaxation} which is a projection on the equilibrium, and one thus recovers a \strong{relaxation scheme}.
    Equilibria are defined by 
    \begin{align*}
        &\vectorial{\distributionFunction}^{\atEquilibrium}_{\zeroVelocity}(\density, \momentumX, \momentumY)
        =
        \begin{pmatrix}
            (1-4\linearEquilibrium_{\density})\density\\
            (1-4\linearEquilibrium_{\momentumX})\momentumX\\
            (1-4\linearEquilibrium_{\momentumY})\momentumY
        \end{pmatrix},
        \\ 
        &\vectorial{\distributionFunction}^{\atEquilibrium}_{\xPlusVelocity}(\density, \momentumX, \momentumY)
        =
        \begin{pmatrix}
            \linearEquilibrium_{\density}\density + \frac{\spaceStep}{2\parabolicScaling}\momentumX\\
            \spaceStep\linearEquilibrium_{\momentumX}\momentumX + \frac{\spaceStep^2}{2\parabolicScaling}\frac{\momentumX^2}{\density} + \frac{1}{2\parabolicScaling} \pressure(\density)\\
            \spaceStep\linearEquilibrium_{\momentumY}\momentumY + \frac{\spaceStep^2}{2\parabolicScaling}\frac{\momentumX\momentumY}{\density}
        \end{pmatrix},\quad
        \vectorial{\distributionFunction}^{\atEquilibrium}_{\xMinusVelocity}(\density, \momentumX, \momentumY)
        =
        \begin{pmatrix}
            \linearEquilibrium_{\density}\density - \frac{\spaceStep}{2\parabolicScaling}\momentumX\\
            \spaceStep\linearEquilibrium_{\momentumX}\momentumX - \frac{\spaceStep^2}{2\parabolicScaling}\frac{\momentumX^2}{\density} - \frac{1}{2\parabolicScaling} \pressure(\density)\\
            \spaceStep\linearEquilibrium_{\momentumY}\momentumY - \frac{\spaceStep^2}{2\parabolicScaling}\frac{\momentumX\momentumY}{\density}
        \end{pmatrix},\\
        &\vectorial{\distributionFunction}^{\atEquilibrium}_{\yPlusVelocity}(\density, \momentumX, \momentumY)
        =
        \begin{pmatrix}
            \linearEquilibrium_{\density}\density + \frac{\spaceStep}{2\parabolicScaling}\momentumY\\
            \spaceStep\linearEquilibrium_{\momentumX}\momentumX + \frac{\spaceStep^2}{2\parabolicScaling}\frac{\momentumX\momentumY}{\density}\\
            \spaceStep\linearEquilibrium_{\momentumY}\momentumY + \frac{\spaceStep^2}{2\parabolicScaling}\frac{\momentumY^2}{\density} + \frac{1}{2\parabolicScaling} \pressure(\density)
        \end{pmatrix},\quad
        \vectorial{\distributionFunction}^{\atEquilibrium}_{\yMinusVelocity}(\density, \momentumX, \momentumY)
        =
        \begin{pmatrix}
            \linearEquilibrium_{\density}\density - \frac{\spaceStep}{2\parabolicScaling}\momentumY\\
            \spaceStep\linearEquilibrium_{\momentumX}\momentumX - \frac{\spaceStep^2}{2\parabolicScaling}\frac{\momentumX\momentumY}{\density}\\
            \spaceStep\linearEquilibrium_{\momentumY}\momentumY - \frac{\spaceStep^2}{2\parabolicScaling}\frac{\momentumY^2}{\density} - \frac{1}{2\parabolicScaling} \pressure(\density)
        \end{pmatrix}.
    \end{align*}
    In these equilibria, we make use of the pressure law $\pressure(\density) = \density^{\gamma}$ with $\gamma\geq 1$.\footnote{Usually, we consider the value $\gamma = 1$.}
    Moreover, we have the real coefficients $\linearEquilibrium_{\density}, \linearEquilibrium_{\momentumX}$, and $\linearEquilibrium_{\momentumY}$ to be chosen.
    \begin{remark}
        Compared to \cite{carfora2008discrete}, we allow different parameters $\linearEquilibrium$ according to the considered conserved quantities.
    \end{remark}
    \item \strong{Transport} phase.
    \begin{align*}
    \vectorial{\distributionFunction}_{\zeroVelocity}(\timeVariable + \timeStep, \xLabel, \yLabel)
    =
    \vectorial{\distributionFunction}_{\zeroVelocity}^{\collided}(\timeVariable, \xLabel, \yLabel), \quad
    &\vectorial{\distributionFunction}_{\xPlusVelocity}(\timeVariable + \timeStep, \xLabel, \yLabel)
    =
    \vectorial{\distributionFunction}_{\xPlusVelocity}^{\collided}(\timeVariable, \xLabel-\spaceStep, \yLabel), \quad 
    \vectorial{\distributionFunction}_{\xMinusVelocity}(\timeVariable + \timeStep, \xLabel, \yLabel)
    =
    \vectorial{\distributionFunction}_{\xMinusVelocity}^{\collided}(\timeVariable, \xLabel+\spaceStep, \yLabel), \\
    &\vectorial{\distributionFunction}_{\yPlusVelocity}(\timeVariable + \timeStep, \xLabel, \yLabel)
    =
    \vectorial{\distributionFunction}_{\yPlusVelocity}^{\collided}(\timeVariable, \xLabel, \yLabel-\spaceStep), \quad 
    \vectorial{\distributionFunction}_{\yMinusVelocity}(\timeVariable + \timeStep, \xLabel, \yLabel)
    =
    \vectorial{\distributionFunction}_{\yMinusVelocity}^{\collided}(\timeVariable, \xLabel, \yLabel+\spaceStep).
\end{align*}
\end{itemize}
Without further mention, initial distribution functions are taken at the equilibrium.

\section{Consistency analysis of the numerical scheme}\label{sec:consistency}

\begin{proposition}\label{prop:Consistency}
    Consider the parabolic scaling \eqref{eq:parabolicScaling}.
    Assume that the numerical solution is obtained as a point-wise sampling of underlying smooth functions of the time and space variables.
    Then, in the limit for $\spaceStep\ll 1$, the underlying smooth function corresponding to the conserved moment $\density$ formally fulfills
    \begin{equation}\label{eq:modifiedEquationDensity}
        \partial_{\timeVariable}
        \density
        +
        \partial_{\xLabel}
        \momentumX
        +
        \partial_{\yLabel}
        \momentumY 
        -
        2\parabolicScaling \linearEquilibrium_{\density}
        \Bigl ( 
            \frac{1}{\relaxationParameter_{\density}}-\frac{1}{2}
        \Bigr )
        (\partial_{\xLabel\xLabel}\density+\partial_{\yLabel\yLabel}\density)
        =\BigO{\spaceStep^2}.
    \end{equation}
    Assuming that
\begin{equation}\label{eq:densityTendingToConstant}
    \pressure(\density(\timeVariable, \xLabel, \yLabel))
    =
    \pressure(\densityIncompressible)
    +
    \spaceStep^2
    \densityIncompressible\pressureLagrange(\timeVariable, \xLabel, \yLabel)
    +\BigO{\spaceStep^3},
\end{equation}
where $\densityIncompressible> 0$ is independent of space and time, the other two conserved moment fulfill
\begin{align}
    \partial_{\timeVariable}\momentumX
    +
    \partial_{\xLabel}
    \Bigl ( \frac{\momentumX^2}{\density} 
         +  \densityIncompressible\pressureLagrange 
    \Bigr )
    +
    \partial_{\yLabel}
    \Bigl ( \frac{\momentumX\momentumY}{\density} 
    \Bigr ) - 2\parabolicScaling \linearEquilibrium_{\momentumX} \Bigl (\frac{1}{\relaxationParameter_{\momentumX}}-\frac{1}{2}\Bigr ) (\partial_{\xLabel\xLabel}\momentumX+\partial_{\yLabel\yLabel}\momentumX) =\BigO{\spaceStep^2}, \label{eq:modifiedEquationMomentumX} \\
    \partial_{\timeVariable}\momentumY
    +
    \partial_{\xLabel}
    \Bigl ( \frac{\momentumX\momentumY}{\density} 
    \Bigr )
    +
    \partial_{\yLabel}
    \Bigl ( \frac{\momentumY^2}{\density} 
         +  \densityIncompressible\pressureLagrange 
    \Bigr )
    - 2\parabolicScaling \linearEquilibrium_{\momentumY} \Bigl (\frac{1}{\relaxationParameter_{\momentumY}}-\frac{1}{2}\Bigr ) (\partial_{\xLabel\xLabel}\momentumY+\partial_{\yLabel\yLabel}\momentumY) =\BigO{\spaceStep^2}. \label{eq:modifiedEquationMomentumY}
\end{align}
\end{proposition}
Notice that thanks to \eqref{eq:densityTendingToConstant}, we obtain that $\density(\timeVariable, \xLabel, \yLabel)
    =
    \densityIncompressible + \BigO{\spaceStep^2}$, where the $\BigO{\spaceStep^2}$-term contains the dependence on time and space.
Writing $\momentumX = \density\velocityX$ and $\momentumY = \density\velocityY$, \eqref{eq:modifiedEquationDensity} becomes 
\begin{equation*}
    \partial_{\xLabel}\velocityX
    +
    \partial_{\yLabel}\velocityY = \BigO{\spaceStep^2},
\end{equation*}
which is the first equation in \eqref{eq:navierStokesSystem} when truncated to the leading order.
We hence see that, provided that \eqref{eq:densityTendingToConstant} holds true, the value of 
\begin{equation*}
    2\parabolicScaling \linearEquilibrium_{\density}
        \Bigl ( 
            \frac{1}{\relaxationParameter_{\density}}-\frac{1}{2}
        \Bigr )
\end{equation*}
does not prevent consistency. Still, this choice can impact accuracy, as well as influencing the stability of the numerical algorithm.
For \eqref{eq:modifiedEquationMomentumX}--\eqref{eq:modifiedEquationMomentumY}, dividing by $\densityIncompressible>0$, we get 
\begin{align*}
    \partial_{\timeVariable}\velocityX
    +
    \partial_{\xLabel}
     ( \velocityX^2
         +  \pressureLagrange 
     )
    +
    \partial_{\yLabel}
     ( \velocityX\velocityY
     ) - 2\parabolicScaling \linearEquilibrium_{\momentumX} \Bigl (\frac{1}{\relaxationParameter_{\momentumX}}-\frac{1}{2}\Bigr ) (\partial_{\xLabel\xLabel}\velocityX+\partial_{\yLabel\yLabel}\velocityX) =\BigO{\spaceStep^2},\\
    \partial_{\timeVariable}\velocityY
    +
    \partial_{\xLabel}
     ( \velocityY\velocityX
     )
    +
    \partial_{\yLabel}
     ( \velocityY^2
         +  \pressureLagrange 
 )
    - 2\parabolicScaling \linearEquilibrium_{\momentumY} \Bigl (\frac{1}{\relaxationParameter_{\momentumY}}-\frac{1}{2}\Bigr ) (\partial_{\xLabel\xLabel}\velocityY+\partial_{\yLabel\yLabel}\velocityY) =\BigO{\spaceStep^2}. 
\end{align*}
Consistency with the second equation in \eqref{eq:navierStokesSystem} is therefore obtained upon having 
\begin{equation*}
    2\parabolicScaling \linearEquilibrium_{\momentumX} \Bigl (\frac{1}{\relaxationParameter_{\momentumX}}-\frac{1}{2}\Bigr )
    =
    2\parabolicScaling \linearEquilibrium_{\momentumY} \Bigl (\frac{1}{\relaxationParameter_{\momentumY}}-\frac{1}{2}\Bigr )
    =
    \viscosity.
\end{equation*}
If we consider the scaling $\parabolicScaling>0$ immutable, the right viscosity $\viscosity$ can be obtained both leveraging the $\linearEquilibrium$'s and the $\relaxationParameter$'s.
Of course, relaxation parameters equal to two are not allowed, as they yield inviscid behavior.
\begin{proof}[Proof of \Cref{prop:Consistency}]
    Let us discuss how to obtain the equation on $\momentumX$: the one on $\momentumY$ is obtained analogously, and the one for $\density$ is even simpler to get.
    We rewrite the numerical scheme using the moments:
\begin{equation*}
    \begin{pmatrix}
        \momentumX\\
        \antiSymmetricMoment_{\xLabel}\\
        \symmetricMoment_{\xLabel}\\
        \antiSymmetricMoment_{\yLabel}\\
        \symmetricMoment_{\yLabel}
    \end{pmatrix}
    =
    \begin{pmatrix}
        \frac{1}{\spaceStep} & \frac{1}{\spaceStep}& \frac{1}{\spaceStep}&\frac{1}{\spaceStep}&\frac{1}{\spaceStep}\\
        0 & 1 & -1 & 0 & 0\\
        0 & 1 & 1 & 0  & 0\\
        0 & 0 & 0 & 1 & -1\\
        0 & 0 & 0 & 1 & 1
    \end{pmatrix}
    \begin{pmatrix}
        {\distributionFunction}_{\zeroVelocity}\\
        {\distributionFunction}_{\xPlusVelocity}\\
        {\distributionFunction}_{\xMinusVelocity}\\
        {\distributionFunction}_{\yPlusVelocity}\\
        {\distributionFunction}_{\yMinusVelocity}
    \end{pmatrix}.
\end{equation*}
We drop subscripts in $\linearEquilibrium_{\xLabel}$ and $\relaxationParameter_{\momentumX}$. 
The relaxation hence becomes 
\begin{align*}
    \momentumX^{\collided} = \momentum, \quad 
    &\antiSymmetricMoment_{\xLabel}^{\collided}
    =
    (1-\relaxationParameter)\antiSymmetricMoment_{\xLabel}
    +
    \relaxationParameter
    \Bigl ( 
        \frac{\spaceStep^2}{\parabolicScaling}\frac{\momentumX^2}{\density} + \frac{1}{\parabolicScaling} \pressure(\density)
    \Bigr ), 
    \quad 
    \symmetricMoment_{\xLabel}^{\collided}
    =
    (1-\relaxationParameter)\symmetricMoment_{\xLabel}
    +
    \relaxationParameter(2
    \spaceStep\linearEquilibrium \momentumX), \\
    &\antiSymmetricMoment_{\yLabel}^{\collided}
    =
    (1-\relaxationParameter) \antiSymmetricMoment_{\yLabel}
    +
    \relaxationParameter
    \Bigl ( \frac{\spaceStep^2}{\parabolicScaling}\frac{\momentumX\momentumY}{\density} \Bigr ), 
    \quad 
    \symmetricMoment_{\yLabel}^{\collided}
    =
    (1-\relaxationParameter) \symmetricMoment_{\yLabel}
    +
    \relaxationParameter(2\spaceStep\linearEquilibrium\momentumX).
\end{align*}
For the transport phase, we perform Taylor expansions when it is written on the moments:
\begin{multline*}\renewcommand*{\arraystretch}{0.5}
    \Bigl (1
    +  
    \frac{\spaceStep^2}{\parabolicScaling}
    \partial_{\timeVariable}
    +
    \BigO{\spaceStep^4}
    \Bigr )
    \begin{pmatrix}
        \momentumX\\
        \antiSymmetricMoment_{\xLabel}\\
        \symmetricMoment_{\xLabel}\\
        \antiSymmetricMoment_{\yLabel}\\
        \symmetricMoment_{\yLabel}
    \end{pmatrix}
    =
    \Biggl [
    \begin{pmatrix}
        1 & -\partial_{\xLabel} & 0 & -\partial_{\yLabel} & 0\\
        0 & 1 & 0 & 0 & 0\\
        0 & 0 & 1 & 0 & 0\\
        0 & 0 & 0 & 1 & 0\\
        0 & 0 & 0 & 0 & 1
    \end{pmatrix}\\
    +\renewcommand*{\arraystretch}{0.5}
    \spaceStep
    \begin{pmatrix}
        0 & 0 & \frac{1}{2}\partial_{\xLabel\xLabel} & 0 & \frac{1}{2}\partial_{\yLabel\yLabel}\\
        0 & 0 & -\partial_{\xLabel} & 0 & 0\\
        0 & -\partial_{\xLabel} &0 & 0 & 0\\
        0 & 0 & 0 & 0 & -\partial_{\yLabel}\\
        0 & 0 & 0 & -\partial_{\yLabel} & 0
    \end{pmatrix}
    +
    \frac{\spaceStep^2}{2}
    \begin{pmatrix}
        0 & -\frac{1}{3}\partial_{\xLabel}^3 & 0 & -\frac{1}{3}\partial_{\yLabel}^3 & 0\\
        0 & \partial_{\xLabel\xLabel} & 0 & 0 & 0\\
        0 & 0 & \partial_{\xLabel\xLabel} & 0 & 0\\
        0 & 0 & 0 & \partial_{\yLabel\yLabel} & 0\\
        0 & 0 & 0 & 0 & \partial_{\yLabel\yLabel}
    \end{pmatrix}\\
    +\renewcommand*{\arraystretch}{0.5}
    \frac{\spaceStep^3}{6}
    \begin{pmatrix}
        0 & 0 & \frac{1}{4}\partial_{\xLabel}^4 & 0 & \frac{1}{4}\partial_{\yLabel}^4\\
        0 & 0 & -\partial_{\xLabel}^3 & 0 & 0\\
        0 & -\partial_{\xLabel}^3 &0 & 0 & 0\\
        0 & 0 & 0 & 0 & -\partial_{\yLabel}^3\\
        0 & 0 & 0 & -\partial_{\yLabel}^3 & 0
    \end{pmatrix}
    +\BigO{\spaceStep^4}
    \Biggr ]
    \begin{pmatrix}
        \momentumX^{\collided}\\
        \antiSymmetricMoment_{\xLabel}^{\collided}\\
        \symmetricMoment_{\xLabel}^{\collided}\\
        \antiSymmetricMoment_{\yLabel}^{\collided}\\
        \symmetricMoment_{\yLabel}^{\collided}
    \end{pmatrix}.
\end{multline*}
Assumption \eqref{eq:densityTendingToConstant} entails, since $\partial_{\xLabel}  \pressure(\densityIncompressible) \equiv 0$, that 
\begin{align*}
    \antiSymmetricMoment_{\xLabel} + \BigO{\spaceStep^2} 
    =
    (1-\relaxationParameter)\antiSymmetricMoment_{\xLabel} 
    +\relaxationParameter
    \pressure(\densityIncompressible)
    -\spaceStep(1-\relaxationParameter)\partial_{\xLabel}\symmetricMoment_{\xLabel}
    + \BigO{\spaceStep^2}, \\
    \symmetricMoment_{\xLabel} + \BigO{\spaceStep^2} 
    =
    (1-\relaxationParameter)\symmetricMoment_{\xLabel}
    +
    \relaxationParameter(2
    \spaceStep\linearEquilibrium \momentumX)
    +\BigO{\spaceStep^2}.
\end{align*}
The second equation entails that $\symmetricMoment_{\xLabel}
    =
    2
    \spaceStep\linearEquilibrium \momentumX
    +
    \BigO{\spaceStep^2}$, hence into the first one: $\antiSymmetricMoment_{\xLabel} = \frac{1}{\parabolicScaling} \pressure(\densityIncompressible) + \BigO{\spaceStep^2}$.
We can go further on $\antiSymmetricMoment_{\xLabel}$ incorporating previous information, and obtain
\begin{equation*}
    \antiSymmetricMoment_{\xLabel}
    +
    \BigO{\spaceStep^4}
    =
    (1-\relaxationParameter)\antiSymmetricMoment_{\xLabel}
    +
    \relaxationParameter
    \Bigl ( 
        \frac{\spaceStep^2}{\parabolicScaling}\frac{\momentumX^2}{\density} + \frac{1}{\parabolicScaling} \pressure(\densityIncompressible)
         + \frac{\spaceStep^2}{\parabolicScaling} \densityIncompressible\pressureLagrange + \BigO{\spaceStep^4}
    \Bigr )
    - 2
    \spaceStep^2\linearEquilibrium \partial_{\xLabel}\momentumX 
    +
    \BigO{\spaceStep^3},
\end{equation*} 
hence 
\begin{equation*}
    \antiSymmetricMoment_{\xLabel}
    =
    \frac{\spaceStep^2}{\parabolicScaling}\frac{\momentumX^2}{\density} + \frac{1}{\parabolicScaling} \pressure(\densityIncompressible)
         + \frac{\spaceStep^2}{\parabolicScaling} \densityIncompressible\pressureLagrange 
        - 
    \spaceStep^2 \frac{2\linearEquilibrium}{\relaxationParameter}  \partial_{\xLabel}\momentumX 
    +
    \BigO{\spaceStep^3}.
\end{equation*}
We obtain analogously 
\begin{equation*}
    \symmetricMoment_{\yLabel}
    =
    2
    \spaceStep\linearEquilibrium \momentumX
    +
    \BigO{\spaceStep^2}
    \quad \text{and}\quad 
    \antiSymmetricMoment_{\yLabel}
    =
    \frac{\spaceStep^2}{\parabolicScaling}\frac{\momentumX\momentumY}{\density} 
        - 
    \spaceStep^2 \frac{2\linearEquilibrium}{\relaxationParameter}  \partial_{\xLabel}\momentumX 
    +
    \BigO{\spaceStep^3}.
\end{equation*}
Into the equation of the conserved moment 
\begin{align*}
    \frac{\spaceStep^2}{\parabolicScaling}
    \partial_{\timeVariable}\momentumX+\BigO{\spaceStep^4}
    =
    &-
    \partial_{\xLabel}
    \Bigl ( 
\frac{\spaceStep^2}{\parabolicScaling}\frac{\momentumX^2}{\density} 
         + \frac{\spaceStep^2}{\parabolicScaling} \densityIncompressible\pressureLagrange 
        - 
    \spaceStep^2 2\linearEquilibrium \Bigl (\frac{1}{\relaxationParameter}-1 \Bigr )  \partial_{\xLabel}\momentumX
    \Bigr ) +
    \spaceStep^2 \linearEquilibrium \partial_{\xLabel\xLabel}\momentumX\\
    &-
    \partial_{\yLabel}
    \Bigl ( \frac{\spaceStep^2}{\parabolicScaling}\frac{\momentumX\momentumY}{\density} 
        - 
    \spaceStep^2 {2\linearEquilibrium}\Bigl (\frac{1}{\relaxationParameter}-1 \Bigr )   \partial_{\xLabel}\momentumX
    \Bigr ) +
    \spaceStep^2 \linearEquilibrium \partial_{\yLabel\yLabel}\momentumX + \BigO{\spaceStep^3},
\end{align*}
hence 
\begin{equation*}
    \partial_{\timeVariable}\momentumX
    +
    \partial_{\xLabel}
    \Bigl ( \frac{\momentumX^2}{\density} 
         +  \densityIncompressible\pressureLagrange 
    \Bigr )
    +
    \partial_{\yLabel}
    \Bigl ( \frac{\momentumX\momentumY}{\density} 
    \Bigr ) - 2\parabolicScaling \linearEquilibrium \Bigl (\frac{1}{\relaxationParameter}-\frac{1}{2}\Bigr ) (\partial_{\xLabel\xLabel}\momentumX+\partial_{\yLabel\yLabel}\momentumX) =\BigO{\spaceStep}.
\end{equation*}
The fact that the reminder $\BigO{\spaceStep}$ is indeed $\BigO{\spaceStep^2}$ can be argued because $\BigO{\timeStep} = \BigO{\spaceStep^2}$ as far as time errors are concerned.
For space errors, we conclude by the fact that---from the symmetry of the discrete velocities---the finite difference operators appearing in the scheme are centered.

\end{proof}

\section{Entropy Dissipation for the LBM with under-relaxation}\label{sec:entropy}

Let us first define the discrete macroscopic quantities, by summing the kinetic distribution $\vectorial{\distributionFunction}_\indexVelocity$ over the discrete velocity space, namely 
\begin{equation}
    \vectorial{W} = \sum_{\indexVelocity\in \{\zeroVelocity, \xPlusVelocity, \yPlusVelocity, \xMinusVelocity, \yMinusVelocity\}} \vectorial{\distributionFunction}_\indexVelocity.
\end{equation}
These correspond to the \strong{conserved quantities} (density and momentums) associated to the target artificial compressible equations, for which we can define a strictly convex macroscopic entropy $\eta(\vectorial{W})$. In the remainder of this section, as done in \cite{carfora2008discrete}, we assume that $\vectorial{W} \in \mathcal{U}$, where $\mathcal{U}$ is an open convex subset of $\reals^3$.

Kinetic entropies $H_{\zeroVelocity}, H_{\xPlusVelocity}, H_{\yPlusVelocity}, H_{\xMinusVelocity}, H_{\yMinusVelocity}$ associated to the LBM scheme are chosen in order to satisfy the following properties.
\begin{enumerate}[label=\textbf{(E$\bm{_\arabic{*}}$)}, ref=\textbf{(E$\bm{_\arabic{*}}$)}]
    \item \label{Kinetic_Entropy_E1} For every $\vectorial{W} \in \mathcal{U}$
        \begin{equation*}
            \sum_{\indexVelocity\in \{\zeroVelocity, \xPlusVelocity, \yPlusVelocity, \xMinusVelocity, \yMinusVelocity\}} H_\indexVelocity(\vectorial{\distributionFunction}^\atEquilibrium_\indexVelocity(\vectorial{U})) = \eta(\vectorial{U}).
        \end{equation*}
    \item \label{Kinetic_Entropy_E2} For every $\vectorial{\distributionFunction}_\indexVelocity$ belonging to the set of equilibria corresponding to $\vectorial{W}\in\mathcal{U}$, with $\mathcal{U} \ni \vectorial{W} = \sum_\indexVelocity \vectorial{\distributionFunction}_\indexVelocity$, it holds
    \begin{equation*}
        \sum_{\indexVelocity\in \{\zeroVelocity, \xPlusVelocity, \yPlusVelocity, \xMinusVelocity, \yMinusVelocity\}} H_\indexVelocity(\vectorial{\distributionFunction}^\atEquilibrium_\indexVelocity(\vectorial{W})) \le \sum_{\indexVelocity\in \{\zeroVelocity, \xPlusVelocity, \yPlusVelocity, \xMinusVelocity, \yMinusVelocity\}} H_\indexVelocity(\vectorial{\distributionFunction}_\indexVelocity).
    \end{equation*}
\end{enumerate}
Let us note that the equilibria that we have selected are compatible in the sense of \cite[Theorem 2.1]{bouchut1999}, which ensures the existence of the kinetic entropies.

\begin{proposition}\label{prop:entropy}
Let $\eta(\vectorial{W})$ be a strictly convex macroscopic entropy function and consider the lattice Boltzmann scheme defined in \eqref{eq:LBM_scheme} with initial condition 
\begin{equation*}
    \vectorial{\distributionFunction}_{\indexVelocity}(0, \xLabel, \yLabel)
    =
    \vectorial{\distributionFunction}_{\indexVelocity}^{\atEquilibrium}(\vectorial{W}(0, \xLabel, \yLabel)),
\end{equation*}
and with all relaxation parameters equal to $\omega$.
Moreover, there exists a strictly convex kinetic entropy functional $\mathcal{H}(\vectorial{\distributionFunction}) := \sum_\indexVelocity H_\indexVelocity(\vectorial{\distributionFunction}_\indexVelocity)$, where $H_\indexVelocity(\vectorial{\distributionFunction}_{\indexVelocity})$ is a strictly convex function for $\vectorial{\distributionFunction}_{\indexVelocity}$ belonging to the set of equilibria corresponding to $\vectorial{W}\in\mathcal{U}$, for each velocity index $\indexVelocity$, satisfying \ref{Kinetic_Entropy_E1}-\ref{Kinetic_Entropy_E2}.\\
Then, if $\omega \in (0, 1]$, the solution of the scheme satisfies a local entropy inequality
  \begin{equation*}
    \mathcal{H}(\vectorial{\distributionFunction}(\timeVariable + \timeStep, \xLabel, \yLabel))
    \leq
    \mathcal{H}(\vectorial{\distributionFunction}(\timeVariable, \xLabel, \yLabel))
    -
    \frac{\timeStep}{\spaceStep}
    \bigl ( 
    \Psi_{\xLabel + \spaceStep, \yLabel}^{\collided}(\timeVariable)  -
    \Psi_{\xLabel - \spaceStep, \yLabel}^{\collided}(\timeVariable)    
    \bigr )
    -
    \frac{\timeStep}{\spaceStep}
    \bigl ( 
    \Psi_{\xLabel, \yLabel+\spaceStep}^{\collided}(\timeVariable)  -
    \Psi_{\xLabel, \yLabel-\spaceStep}^{\collided}(\timeVariable)    
    \bigr ),
  \end{equation*}
where the fluxes are given by
  \begin{multline}\label{eq:entropyFluxX}
    \Psi_{\xLabel + \spaceStep, \yLabel}^{\collided}(\timeVariable)
    :=
    \parabolicScaling
    \frac{H_{\xMinusVelocity}(\vectorial{\distributionFunction}_{\xMinusVelocity}^{\collided}(\timeVariable, \xLabel, \yLabel)) - H_{\xMinusVelocity}(\vectorial{\distributionFunction}_{\xMinusVelocity}^{\collided}(\timeVariable, \xLabel+\spaceStep, \yLabel))}{\spaceStep}, 
    \\ 
    \Psi_{\xLabel - \spaceStep, \yLabel}^{\collided}(\timeVariable)
    :=
    \parabolicScaling
    \frac{H_{\xPlusVelocity}(\vectorial{\distributionFunction}_{\xPlusVelocity}^{\collided}(\timeVariable, \xLabel-\spaceStep, \yLabel)) - H_{\xPlusVelocity}(\vectorial{\distributionFunction}_{\xPlusVelocity}^{\collided}(\timeVariable, \xLabel, \yLabel))}{\spaceStep}
  \end{multline}
  and
  \begin{multline}\label{eq:entropyFluxY}
    \Psi_{\xLabel , \yLabel + \spaceStep}^{\collided}(\timeVariable)
    :=
    \parabolicScaling
    \frac{H_{\yMinusVelocity}(\vectorial{\distributionFunction}_{\yMinusVelocity}^{\collided}(\timeVariable, \xLabel, \yLabel)) - H_{\yMinusVelocity}(\vectorial{\distributionFunction}_{\yMinusVelocity}^{\collided}(\timeVariable, \xLabel, \yLabel+\spaceStep))}{\spaceStep}, 
    \\ 
    \Psi_{\xLabel , \yLabel- \spaceStep}^{\collided}(\timeVariable)
    :=
    \parabolicScaling
    \frac{H_{\yPlusVelocity}(\vectorial{\distributionFunction}_{\yPlusVelocity}^{\collided}(\timeVariable, \xLabel, \yLabel-\spaceStep)) - H_{\yPlusVelocity}(\vectorial{\distributionFunction}_{\yPlusVelocity}^{\collided}(\timeVariable, \xLabel, \yLabel))}{\spaceStep}.
  \end{multline}
  Moreover, we have the following estimate on the decrease of the total entropy in the domain:
 \begin{equation}
 \label{eq:entropy_inequality}
     \sum_{(\xLabel,\yLabel)\in (\relatives+\frac{1}{2})^2}
    \eta (\vectorial{W}(\timeVariable, \xLabel, \yLabel))\leq 
    \sum_{(\xLabel,\yLabel)\in (\relatives+\frac{1}{2})^2}
    \eta (\vectorial{W}(0, \xLabel, \yLabel)).
 \end{equation}
 \end{proposition}
  
\begin{proof}
  Since $\omega \in (0, 1]$, the post-collision state $\vectorial{\distributionFunction}_\indexVelocity^{{\collided}}(\timeVariable, \xLabel, \yLabel)$ represents a strict convex combination of $\vectorial{f}_\indexVelocity (\timeVariable, \xLabel, \yLabel)$ and the equilibrium state $\vectorial{\distributionFunction}^{\atEquilibrium}_\indexVelocity(\vectorial{W}(\timeVariable, \xLabel, \yLabel))$. 
  Given that the kinetic entropy density function $H_\indexVelocity$ is strictly convex in its argument, we apply Jensen's inequality and obtain
  \begin{equation*}
      H_\indexVelocity(\vectorial{\distributionFunction}_\indexVelocity^{\collided}(\timeVariable, \xLabel, \yLabel)) \le (1-\omega) H_\indexVelocity(\vectorial{\distributionFunction}_\indexVelocity (\timeVariable, \xLabel, \yLabel)) + \omega H(\vectorial{\distributionFunction}^{\atEquilibrium}_\indexVelocity(\vectorial{W}(\timeVariable, \xLabel, \yLabel))).
  \end{equation*}
  Summing this relation over the velocity space yields the inequality for the total local kinetic entropy
  \begin{equation*}
      \mathcal{H}(\vectorial{\distributionFunction}^{\collided}(\timeVariable, \xLabel, \yLabel)) \le (1-\omega) \mathcal{H}(\vectorial{\distributionFunction}(\timeVariable, \xLabel, \yLabel)) + \omega \mathcal{H}(\vectorial{\distributionFunction}^{\atEquilibrium}(\vectorial{W}(\timeVariable, \xLabel, \yLabel))).
  \end{equation*}
  Since the equilibrium is a minimizer for the entropy $\mathcal{H}$ among all distributions possessing the same macroscopic moments, see \ref{Kinetic_Entropy_E2}, we obtain
  \begin{equation}\label{eq:entropyDimCollision}
      \mathcal{H}(\vectorial{\distributionFunction}^{\collided}(\timeVariable, \xLabel, \yLabel)) \le  \mathcal{H}(\vectorial{\distributionFunction}(\timeVariable, \xLabel, \yLabel)).
  \end{equation}
  This proves that under-relaxation $\omega \leq 1$ guarantees that the total kinetic entropy is decreasing through the relaxation phase. The streaming phase represents a shift of $\vectorial{\distributionFunction}^{\collided}$ along the discrete lattice. 
  In particular, we have 
  \begin{align*}
    \mathcal{H}(\vectorial{\distributionFunction}(\timeVariable + \timeStep, \xLabel, \yLabel))
    =
    H_{\zeroVelocity}(\vectorial{\distributionFunction}_{\zeroVelocity}^{\collided}(\timeVariable, \xLabel, \yLabel))
    &+
    H_{\xPlusVelocity}(\vectorial{\distributionFunction}_{\xPlusVelocity}^{\collided}(\timeVariable, \xLabel-\spaceStep, \yLabel))
    +
    H_{\xMinusVelocity}(\vectorial{\distributionFunction}_{\xMinusVelocity}^{\collided}(\timeVariable, \xLabel+\spaceStep, \yLabel))\\
    &+
    H_{\yPlusVelocity}(\vectorial{\distributionFunction}_{\yPlusVelocity}^{\collided}(\timeVariable, \xLabel, \yLabel-\spaceStep))
    +
    H_{\yMinusVelocity}(\vectorial{\distributionFunction}_{\yMinusVelocity}^{\collided}(\timeVariable, \xLabel, \yLabel+\spaceStep))
.
  \end{align*}
  Straightforward manipulations turn the previous inequality into
  \begin{align*}
    \mathcal{H}(\vectorial{\distributionFunction}(\timeVariable + \timeStep, \xLabel, \yLabel))
    &=
    \mathcal{H}(\vectorial{\distributionFunction}^{\collided}(\timeVariable, \xLabel, \yLabel))
    -
    \frac{\timeStep}{\spaceStep}
    \bigl ( 
    \Psi_{\xLabel + \spaceStep, \yLabel}^{\collided}(\timeVariable)  -
    \Psi_{\xLabel - \spaceStep, \yLabel}^{\collided}(\timeVariable)    
    \bigr )
    -
    \frac{\timeStep}{\spaceStep}
    \bigl ( 
    \Psi_{\xLabel, \yLabel+\spaceStep}^{\collided}(\timeVariable)  -
    \Psi_{\xLabel, \yLabel-\spaceStep}^{\collided}(\timeVariable)    
    \bigr )\\
    &\leq
    \mathcal{H}(\vectorial{\distributionFunction}(\timeVariable, \xLabel, \yLabel))
    -
    \frac{\timeStep}{\spaceStep}
    \bigl ( 
    \Psi_{\xLabel + \spaceStep, \yLabel}^{\collided}(\timeVariable)  -
    \Psi_{\xLabel - \spaceStep, \yLabel}^{\collided}(\timeVariable)    
    \bigr )
    -
    \frac{\timeStep}{\spaceStep}
    \bigl ( 
    \Psi_{\xLabel, \yLabel+\spaceStep}^{\collided}(\timeVariable)  -
    \Psi_{\xLabel, \yLabel-\spaceStep}^{\collided}(\timeVariable)    
    \bigr ),
  \end{align*}
  where the inequality comes from \eqref{eq:entropyDimCollision}, and the fluxes are defined by \eqref{eq:entropyFluxX} and \eqref{eq:entropyFluxY}.
  Summing in $\xLabel$ and $\yLabel$ over the discrete mesh, we obtain 
\begin{align*}
    \sum_{(\xLabel,\yLabel)\in (\relatives+\frac{1}{2})^2}
    \mathcal{H}(\vectorial{\distributionFunction}(\timeVariable, \xLabel, \yLabel))
    &\leq 
     \sum_{(\xLabel,\yLabel)\in (\relatives+\frac{1}{2})^2}
    \mathcal{H}(\vectorial{\distributionFunction}(\timeVariable - \timeStep, \xLabel, \yLabel))
    \leq \dots \leq 
    \sum_{(\xLabel,\yLabel)\in (\relatives+\frac{1}{2})^2}
    \mathcal{H}(\vectorial{\distributionFunction}(0, \xLabel, \yLabel))\\
    &= 
    \sum_{(\xLabel,\yLabel)\in (\relatives+\frac{1}{2})^2}
    \mathcal{H}(\vectorial{\distributionFunction}^{\atEquilibrium}(\vectorial{W}(0, \xLabel, \yLabel)))
    =
    \sum_{(\xLabel,\yLabel)\in (\relatives+\frac{1}{2})^2}
    \eta (\vectorial{W}(0, \xLabel, \yLabel)),
  \end{align*}
  where the first inequality comes from the conservativity of the fluxes that yields cancellations.
  We then iterate across time steps, use initialization at equilibrium, and finally \ref{Kinetic_Entropy_E1}.
  This gives, thanks to \ref{Kinetic_Entropy_E2}, that for $\timeVariable\in\timeStep\naturals$
  \begin{equation*}
    \sum_{(\xLabel,\yLabel)\in (\relatives+\frac{1}{2})^2}
    \eta (\vectorial{W}(\timeVariable, \xLabel, \yLabel))\leq 
    \sum_{(\xLabel,\yLabel)\in (\relatives+\frac{1}{2})^2}
    \eta (\vectorial{W}(0, \xLabel, \yLabel)).
  \end{equation*}

\end{proof}

\section{A partial study of the spectra for the linearized scheme}\label{sec:spectra}

We now draw some partial conclusions from spectra when linearizing the numerical scheme.
This partial character is due to the large number of eigenvalues when not relaxing on the equilibrium.

Consider now the linear pressure law $\pressure(\density) = \density$.
We linearize the equilibria, thus the entire scheme, about a reference state $(\referenceState{\density}, \referenceState{\momentumX}, \referenceState{\momentumY})$.
Moreover, one considers the Fourier transform in space, denoted by a hat, giving 
\begin{equation*}
    \fourierTransformed{\vectorial{\distributionFunction}}
    (\timeVariable+\timeStep, \frequency_{\xLabel}, \frequency_{\yLabel})
    =
    \vectorial{E}_{(\referenceState{\density}, \referenceState{\momentumX}, \referenceState{\momentumY})}
    (\frequency_{\xLabel}\spaceStep, \frequency_{\yLabel}\spaceStep)
    \fourierTransformed{\vectorial{\distributionFunction}}
    (\timeVariable, \frequency_{\xLabel}, \frequency_{\yLabel})
\end{equation*}
where the distribution functions have been collected in a vector, so that $\vectorial{E}_{(\referenceState{\density}, \referenceState{\momentumX}, \referenceState{\momentumY})}\in\mathcal{M}_{15}(\mathbb{C})$, and $(\frequency_{\xLabel}\spaceStep, \frequency_{\yLabel}\spaceStep)\in[-\pi, \pi]^2$.

\subsection{Checkerboard mode}
We first recover a necessary stability condition by analyzing the checkerboard mode $(\pi, \pi)$.
\begin{lemma}[Necessary stability condition]\label{lemma:spectrumPiPi}
    Necessary and sufficient conditions so that the spectrum of $\vectorial{E}_{(\referenceState{\density}, \referenceState{\momentumX}, \referenceState{\momentumY})}(\pi, \pi)$ belongs to the closed unit disk are that 
    \begin{equation*}
        \relaxationParameter_{\density}, \relaxationParameter_{\momentumX}, \relaxationParameter_{\momentumY} \in [0, 2]
        \qquad \text{and}\qquad 
        \linearEquilibrium_{\density}, \linearEquilibrium_{\momentumX}, \linearEquilibrium_{\momentumY} \in [0, \tfrac{1}{4}].
    \end{equation*}
\end{lemma}
\begin{proof}
    Computations yield that 
    \begin{equation*}
        \textnormal{det}(z\vectorial{I}-\vectorial{E}_{(\referenceState{\density}, \referenceState{\momentumX}, \referenceState{\momentumY})}(\pi, \pi))
        =
        \prod_{i \in\{ \density, \momentumX, \momentumY\}}
        (z-\relaxationParameter_i+1)^3
        (z^2-\relaxationParameter_i(1-8\linearEquilibrium_i)z + \relaxationParameter_i - 1),
    \end{equation*}
    hence the condition follows from the recurrent procedure exposed in \cite[Chapter 4]{strikwerda2004finite}.
\end{proof}

\begin{remark}[Role of $\spaceStep$ in the spectrum]
    The linearized equilibria depend on $\spaceStep$.
    By numerically computing the spectrum of $\vectorial{E}_{(\referenceState{\density}, \referenceState{\momentumX}, \referenceState{\momentumY})}
    (\frequency_{\xLabel}\spaceStep, \frequency_{\yLabel}\spaceStep)$, we see that it stabilizes for $\spaceStep\to 0$.
    For this reason, we present results with $\spaceStep=10^{-6}$.
\end{remark}

\subsection{Role of the space-time scaling $\parabolicScaling$}

\begin{figure}
    \begin{center}
        \includegraphics[width=1\textwidth]{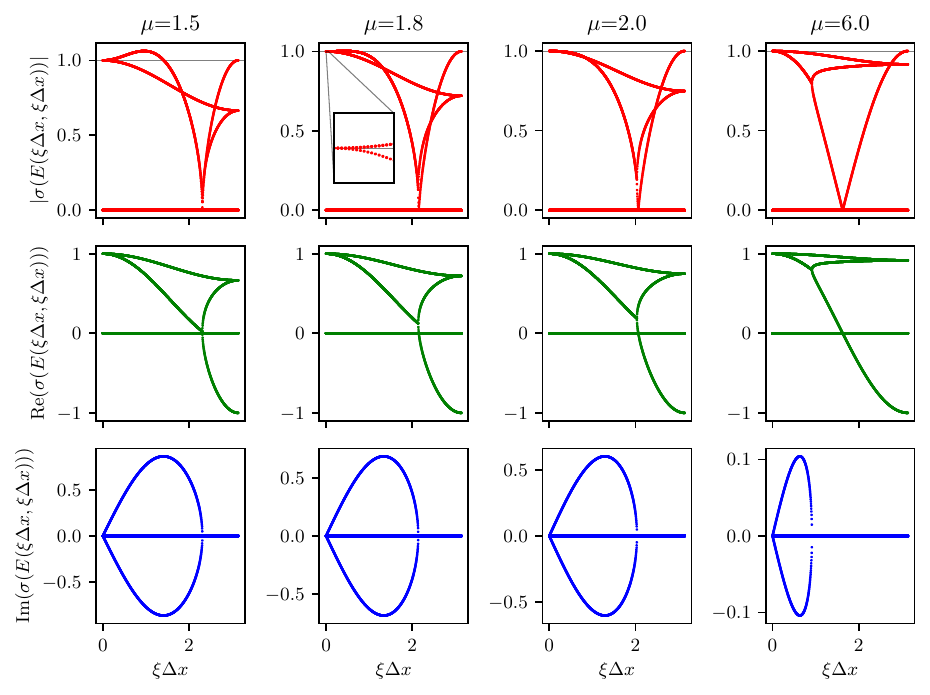}
    \end{center}
    \caption{\label{fig:roleMu}Modulus, real and imaginary part of the eigenvalues of $\vectorial{E}_{(\referenceState{\density}, \referenceState{\momentumX}, \referenceState{\momentumY})}
    (\frequency\spaceStep, \frequency\spaceStep)$, that is ``along a diagonal'' in the frequency space.}
\end{figure}

Let us start by a numerical illustration, in which we linearize around $(\referenceState{\density}, \referenceState{\momentumX}, \referenceState{\momentumY}) = (1, 1, 1)$.
We consider the relaxation scheme with $\relaxationParameter_{\density} = \relaxationParameter_{\momentumX} = \relaxationParameter_{\momentumY} = 1$.
Moreover, we select $\linearEquilibrium_{\density}=\tfrac{1}{4}$, and 
\begin{equation*}
    \viscosity = \frac{\pi}{50}, 
    \qquad \text{thus}
    \qquad 
    \linearEquilibrium_{\momentumX}
    =
    \linearEquilibrium_{\momentumY}
    =
    \frac{\viscosity}{\parabolicScaling},
\end{equation*}
to achieve the viscosity.
Modulii of spectra while varying $\parabolicScaling$ are shown in \Cref{fig:roleMu}.
We see that a minimum $\parabolicScaling$ to achieve stability is needed, as it ensures that one of the three present modes dissipates ``in the right direction'' in the low frequency limit.
Moreover, we see that the fact of having taken $\linearEquilibrium_{\density}=\tfrac{1}{4}$ yields---as visible from \Cref{lemma:spectrumPiPi}---a lack of damping of the checkerboard mode $(\pi, \pi)$, which can be dangerous in the non-linear setting.

In order to avoid dealing with an equation of order 15, we consider the relaxation setting $\relaxationParameter_{\density} = \relaxationParameter_{\momentumX} = \relaxationParameter_{\momentumY} = 1$.
We want to find a necessary stability condition in terms of $\parabolicScaling$ by requesting that eigenvalues to not grow in modulus above one close to the zero frequency.
We conjecture, based on numerical evidence, that the condition be slightly different for out-of-equilibrium schemes.
\begin{lemma}
    Consider the linearized relaxation scheme ($\relaxationParameter_{\density} = \relaxationParameter_{\momentumX} = \relaxationParameter_{\momentumY} = 1$). 
    Then, a necessary stability condition is that 
    \begin{equation*}
    \parabolicScaling
    \geq 
    (\linearEquilibrium_{\density} + \linearEquilibrium_{\momentumX}  +\linearEquilibrium_{\momentumY})^{-1/2}.
\end{equation*}
\end{lemma}
\begin{proof}
Computations give
\begin{multline*}
    \textnormal{det}(z-\vectorial{E}_{(\referenceState{\density}, \referenceState{\momentumX}, \referenceState{\momentumY})}
    (\frequency\spaceStep, \frequency\spaceStep)) 
    = 
    \bigl (2  {\linearEquilibrium_{\density}}  (\frequency\spaceStep)^{2} z + 2  {\linearEquilibrium_{\momentumX}}  (\frequency\spaceStep)^{2} z + 2  {\linearEquilibrium_{\momentumY}}  (\frequency\spaceStep)^{2} z \\
    - 2  {\linearEquilibrium_{\density}}  (\frequency\spaceStep)^{2} - 2  {\linearEquilibrium_{\momentumX}}  (\frequency\spaceStep)^{2} - 2  {\linearEquilibrium_{\momentumY}}  (\frequency\spaceStep)^{2} +  z^{2} - 2   z +  2/\parabolicScaling^{2}  (\frequency\spaceStep)^{2} \bigr ) {\left(z - 1\right)} z^{12}
    +\BigO{(\frequency\spaceStep)^3},
\end{multline*}    
where we have performed a second-order Taylor expansion in the coefficients in the limit of $\frequency\spaceStep\to 0$, and we have set $\spaceStep = 0$ in the coefficients.
As expected, three eigenvalues equal one at leading-order in $\frequency\spaceStep$. 
One of these does not represent a propagating mode.
We now follow the remaining two: taking $z = 1+(\frequency\spaceStep) z^{(1)} + \BigO{(\frequency\spaceStep)^2}$, and truncating---yields the quadratic equation $\parabolicScaling^2 (z^{(1)})^2 +2 = 0$, hence $z^{(1)} = \pm i \frac{\sqrt{2}}{\parabolicScaling}$.
This means 
\begin{equation*}
    z = 1\pm i \frac{\sqrt{2}}{\parabolicScaling} (\frequency\spaceStep) +\BigO{(\frequency\spaceStep)^2}
\end{equation*}
and thus low spatial frequencies for these two waves propagate.
Going further and taking $z = 1 \pm i \frac{\sqrt{2}}{\parabolicScaling} (\frequency\spaceStep) + (\frequency\spaceStep)^2 z^{(2)} + \BigO{(\frequency\spaceStep)^3}$ gives 
\begin{equation*}
    z = 1\pm i \frac{\sqrt{2}}{\parabolicScaling} (\frequency\spaceStep) 
    -(\linearEquilibrium_{\density} + \linearEquilibrium_{\momentumX}  +\linearEquilibrium_{\momentumY})(\frequency\spaceStep)^2
    +\BigO{(\frequency\spaceStep)^3}.
\end{equation*}
We are now in position to request that these two modes ``dissipate'' in the low spatial frequency limit. 
Indeed 
\begin{align*}
    |z|^{2}
    &=
    (1-(\linearEquilibrium_{\density} + \linearEquilibrium_{\momentumX}  +\linearEquilibrium_{\momentumY})(\frequency\spaceStep)^2 + \BigO{(\frequency\spaceStep)^4})^2
    + (\pm \frac{\sqrt{2}}{\parabolicScaling} (\frequency\spaceStep) +  \BigO{(\frequency\spaceStep)^3})^2\\
    &=
    1 
    - 2\Bigl( 
    \linearEquilibrium_{\density} + \linearEquilibrium_{\momentumX}  +\linearEquilibrium_{\momentumY} - \frac{1}{\parabolicScaling^2}    
    \Bigr ) (\frequency\spaceStep)^2 + 
    \BigO{(\frequency\spaceStep)^4},
\end{align*}
which gives $\parabolicScaling
    >
    (\linearEquilibrium_{\density} + \linearEquilibrium_{\momentumX}  +\linearEquilibrium_{\momentumY})^{-1/2}$.
    More generally, if we look for the spectrum where along the direction $(\cos(\beta), \sin(\beta))$ in the frequency space, we obtain 
    \begin{equation*}
    z = 1\pm i \frac{1}{\parabolicScaling} (\frequency\spaceStep) 
    -\frac{1}{2}(\linearEquilibrium_{\density} + \linearEquilibrium_{\momentumX}  +\linearEquilibrium_{\momentumY})(\frequency\spaceStep)^2
    +\BigO{(\frequency\spaceStep)^3},
\end{equation*}
which means that acoustic waves propagate (quickly) at velocity $1/\spaceStep$ and gives an analogous result.
\end{proof}

\section{Numerical experiments}\label{sec:numericalExp}

We now gather numerical experiments to validate the proposed approach and to study how to select the parameters in the numerical scheme.
Let us stress that we consider simulations with relaxation parameters in $[1, 2)$, which generally ensure better performances.
However, one is able to straightforwardly prove discrete entropy inequalities, see \Cref{prop:entropy}, only when relaxation parameters belong to $(0, 1]$.

\subsection{Taylor-Green vortex}

We consider precisely the setting of Section 5 in \cite{carfora2008discrete} with a Taylor-Green vortex. 
In this case, we fix $\parabolicScaling = 8$.

\subsubsection{Dissipation on the density $\density$}

\begin{figure}
    \begin{center}
        \includegraphics[width=.5\textwidth]{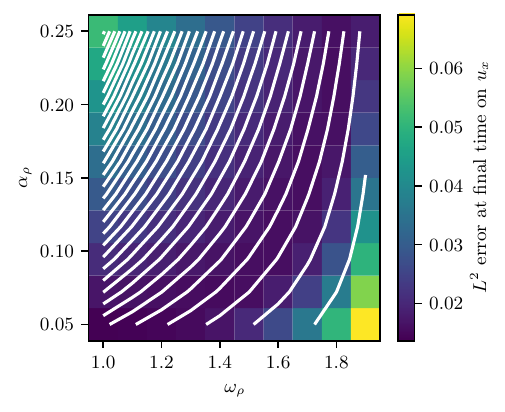}
    \end{center}
    \caption{\label{fig:taylorGreen-error-dissipation-rho}Error at final time for the Taylor-Green vortex test as function of $\relaxationParameter_{\density}$ and $\linearEquilibrium_{\density}$.
    White lines: contours of the function $\linearEquilibrium_{\density}
         ( 
            \frac{1}{\relaxationParameter_{\density}}-\frac{1}{2}
         )$.}
\end{figure}

We have seen from \Cref{prop:Consistency} that the density field $\density$ is diffused, with diffusion coefficient 
\begin{equation*}
    \text{proportional to}\quad  
     \linearEquilibrium_{\density}
        \Bigl ( 
            \frac{1}{\relaxationParameter_{\density}}-\frac{1}{2}
        \Bigr ).
\end{equation*} 
The question is: how does the choice of diffusion on $\density$, depending on $\linearEquilibrium_{\density}$ and $\relaxationParameter_{\density}$, relate to the accuracy of the numerical solution?

To provide a possible answer to this question, we simulate with 75 cells-per-direction, and measure $L^2$ errors on the velocity along $\xLabel$ at final time equal to one.
In this case, we use $\relaxationParameter_{\momentumX} = \relaxationParameter_{\momentumY} = 1$ (relaxation scheme for the two momentums), and obtain the requested the viscosity using $\linearEquilibrium_{\momentumX} = \linearEquilibrium_{\momentumY}$.
Results are collected in \Cref{fig:taylorGreen-error-dissipation-rho}.
We see that better results can be achieved by reducing the diffusion on $\density$, of course taking into account the lack of numerical stability when this feature is exacerbated (see for example the lower-right corner).
There is a whole curve---not far from being an isoline of the diffusivity---in the $(\relaxationParameter_{\density}, \linearEquilibrium_{\density})$-plane where  minimal error can be achieved, and this curve includes the relaxation scheme $\relaxationParameter_{\density} = 1$ with $\linearEquilibrium_{\density}$ far away from $\tfrac{1}{4}$, which was the value considered in \cite{carfora2008discrete}.

\subsubsection{Relaxation parameters for the momentums}

\begin{figure}
    \begin{center}
        \includegraphics[width=.5\textwidth]{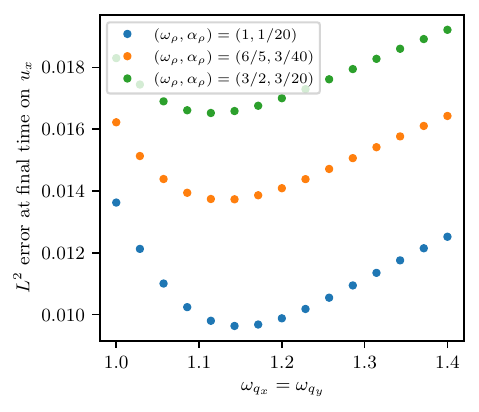}
    \end{center}
    \caption{\label{fig:taylorGreen-error-omega_q}Error at final time for the Taylor-Green vortex test as function of $\relaxationParameter_{\momentumX} = \relaxationParameter_{\momentumY}$.}
\end{figure}

We now consider the following values for $(\relaxationParameter_{\density}, \linearEquilibrium_{\density})$ that lay on the same isoline of $\linearEquilibrium_{\density}
        \Bigl ( 
            \frac{1}{\relaxationParameter_{\density}}-\frac{1}{2}
        \Bigr )$ and which roughly minimize the error in \Cref{fig:taylorGreen-error-dissipation-rho}:
\begin{equation}\label{eq:params}
    (\relaxationParameter_{\density}, \linearEquilibrium_{\density}) 
    \in
    \{
    (1, 1/20), 
    (6/5, 3/40), 
    (3/2, 3/20)  
    \}.
\end{equation}
For these values, we check the $L^2$ error on the $\xLabel$ velocity at final time varying $\relaxationParameter_{\momentumX} = \relaxationParameter_{\momentumY}$ ($\linearEquilibrium_{\momentumX} = \linearEquilibrium_{\momentumX}$ is adjusted to enforce the right diffusivity (Reynolds number)).
The results in \Cref{fig:taylorGreen-error-omega_q} show that in this setting, there is a slight advantage in taking $\relaxationParameter_{\momentumX} = \relaxationParameter_{\momentumY}$ above one, let us say around $1.15$.

\subsubsection{Convergence}

\begin{table}
    \begin{center}\caption{
        \label{tab:convergenceTaylorGreen}
        Empirical convergence for the Taylor-Green vortex test for different choices of parameters.
        Errors are computed at final time equal to one.}
        \begin{tabular}{|cc||c|c||c|c||c|c|}
            \hline
                & 
                & \multicolumn{2}{c||}{$(\relaxationParameter_{\density}, \linearEquilibrium_{\density}) = (1, 1/20)$}
                & \multicolumn{2}{c||}{$(\relaxationParameter_{\density}, \linearEquilibrium_{\density}) = (6/5, 3/40)$}
                & \multicolumn{2}{c|}{$(\relaxationParameter_{\density}, \linearEquilibrium_{\density}) = (3/2, 3/20)$}\\
                \hline
                cells-per-dir. & $\spaceStep$ 
                & $L^2$ error on $\velocityX$ & order 
                & $L^2$ error on $\velocityX$ & order 
                & $L^2$ error on $\velocityX$ & order\\
                \hline
        75	& 8.378E-02	& 9.626E-03 &	---			    &   1.375E-02	& ---		    &   1.662E-02	& ---\\      
        112	& 5.610E-02	& 3.719E-03 &	2.37			&	5.593E-03	& 2.24			&	6.872E-03	& 2.20\\
        168	& 3.740E-02	& 2.001E-03 &	1.53			&	2.969E-03	& 1.56			&	3.555E-03	& 1.63\\
        253	& 2.483E-02	& 7.270E-04 &	2.47			&	1.029E-03	& 2.59			&	1.264E-03	& 2.53\\
        379	& 1.658E-02	& 4.231E-04 &	1.34			&	6.201E-04	& 1.25			&	7.393E-04	& 1.33\\
        \hline 
        Average & --- & --- & 1.93 & --- & 1.91 & --- & 1.92 \\
        \hline
    \end{tabular}
    \end{center}
\end{table}

We now make the number of grid-points increase and use the following sets of parameters in \eqref{eq:params} along with $\relaxationParameter_{\momentumX} = \relaxationParameter_{\momentumY} = 23/20 = 1.15$.
The results in \Cref{tab:convergenceTaylorGreen} provide empirical second-order convergence of the scheme in the velocity variable, analogously to \cite{carfora2008discrete}.

\subsection{Poiseuille flow}

We consider a two-dimensional Poiseuille flow in the square $[0, L]^2$, as in \cite[Section 6.1.2]{bouchut2018}.
Given a maximal velocity $U>0$ along the $\xLabel$-axis and a target Reynolds number $\textnormal{Re}>0$, so that the viscosity is $\viscosity = U L /\textnormal{Re}$, the exact (steady) profile is given by 
\begin{equation*}
    \velocityX^{\textnormal{ex}}
    =
    \frac{4U}{L^2}\yLabel(L-\yLabel), 
    \qquad 
    \velocityY^{\textnormal{ex}}
    =0, 
    \qquad 
    \pressureLagrange^{\textnormal{ex}}
    =
    \frac{8\viscosity U}{L^2}(L-\xLabel).
\end{equation*}
In the numerical tests, we select $L = 1$, $U = 1$, and $\textnormal{Re} = 100$.

\subsubsection{Numerical boundary conditions}

Let us now precisely describe how we implement boundary conditions.
The idea is to approximate the conserved quantities $(\density, \momentumX, \momentumY)$ in ghost cells around the domain using second-order formul\ae{} \cite{bouchut2018} (for instance, extrapolation, when needed), and then employ these data in the equilibria of the (lacking) distribution function to be streamed inside the domain.
This approach and the possibility of having convergence for lattice Boltzmann schemes in the scalar case has been investigated in \cite{aregba2025equilibrium}.
This way of proceeding is very convenient, since we devise boundary conditions on ``physical'' quantities we are interested in, and finally use equilibria to construct distribution functions from them.
However, we shall see that when the numerical scheme is genuinely lattice Boltzmann (not a relaxation scheme), this can cause order reduction to one.
We explain how this can be corrected.

\begin{itemize}
    \item Left boundary (at $\xLabel = 0$). For every needed $\yLabel\in\spaceStep(\relatives+\tfrac{1}{2})$:
    \begin{align*}
        \density(\timeVariable, -\tfrac{\spaceStep}{2}, \yLabel)
        &:=
        2 \density(\timeVariable, \tfrac{\spaceStep}{2}, \yLabel)
        -
        \density(\timeVariable, \tfrac{3\spaceStep}{2}, \yLabel),\\
        \momentumX(\timeVariable, -\tfrac{\spaceStep}{2}, \yLabel)
        &:=
        2\densityIncompressible\velocityX^{\textnormal{ex}}(\yLabel)
        -
        \momentumX(\timeVariable, \tfrac{\spaceStep}{2}, \yLabel), \\
        \momentumY(\timeVariable, -\tfrac{\spaceStep}{2}, \yLabel)
        &:=
        2\densityIncompressible\velocityY^{\textnormal{ex}}
        -
        \momentumY(\timeVariable, \tfrac{\spaceStep}{2}, \yLabel).
    \end{align*}
    The aim is to impose the exact velocity profile.
    Notice that the second-order extrapolation on $\density$ is done on points separated by $\spaceStep$, whereas the one for the momentums concerns points with distance $\frac{\spaceStep}{2}$, and we force the exact velocity profile and the ``reference'' density $\densityIncompressible$ at $\xLabel = 0$.
    Then, we prepare the post-relaxation ghost values as 
    \begin{equation}\label{eq:BCPoiseuilleLeftEquilibrium}
        \vectorial{\distributionFunction}_{\xPlusVelocity}^{\collided}(\timeVariable, -\tfrac{\spaceStep}{2}, \yLabel)
        :=
        \vectorial{\distributionFunction}_{\xPlusVelocity}^{\atEquilibrium}
        (\density(\timeVariable, -\tfrac{\spaceStep}{2}, \yLabel), \momentumX(\timeVariable, -\tfrac{\spaceStep}{2}, \yLabel), \momentumY(\timeVariable, -\tfrac{\spaceStep}{2}, \yLabel)).
    \end{equation}
    \item Right boundary (at $\xLabel = L$). For every $\yLabel\in\spaceStep(\relatives+\tfrac{1}{2})$:
    \begin{align*}
        \density(\timeVariable, L+\tfrac{\spaceStep}{2}, \yLabel)
        &:=
        2\densityIncompressible(1+\spaceStep^2 \pressureLagrange^{\textnormal{ex}}(L))
        -\density(\timeVariable, L-\tfrac{\spaceStep}{2}, \yLabel)
        ,\\
        \momentumX(\timeVariable, L+\tfrac{\spaceStep}{2}, \yLabel)
        &:=
        2\momentumX(\timeVariable, L-\tfrac{\spaceStep}{2}, \yLabel)
        -
        \momentumX(\timeVariable, L-\tfrac{3\spaceStep}{2}, \yLabel), \\
        \momentumY(\timeVariable, L+\tfrac{\spaceStep}{2}, \yLabel)
        &:=
        2\momentumY(\timeVariable, L-\tfrac{\spaceStep}{2}, \yLabel)
        -
        \momentumY(\timeVariable, L-\tfrac{3\spaceStep}{2}, \yLabel).
    \end{align*}
    The aim is to impose the exact pressure (equal to zero), using the truncated relation \eqref{eq:densityTendingToConstant}.
    Then
    \begin{equation}\label{eq:BCPoiseuilleRightEquilibrium}
        \vectorial{\distributionFunction}_{\xMinusVelocity}^{\collided}(\timeVariable, L+\tfrac{\spaceStep}{2}, \yLabel)
        :=
        \vectorial{\distributionFunction}_{\xMinusVelocity}^{\atEquilibrium}
        (\density(\timeVariable, L+\tfrac{\spaceStep}{2}, \yLabel), \momentumX(\timeVariable, L+\tfrac{\spaceStep}{2}, \yLabel), \momentumY(\timeVariable, L+\tfrac{\spaceStep}{2}, \yLabel)).
    \end{equation}
    \item Lower boundary (at $\yLabel = 0$). For every $\xLabel\in\spaceStep(\relatives+\tfrac{1}{2})$:
    \begin{align*}
        \density(\timeVariable, \xLabel, -\tfrac{\spaceStep}{2})
        &:=
        2\density(\timeVariable, \xLabel, \tfrac{\spaceStep}{2})
        -\density(\timeVariable, \xLabel, \tfrac{3\spaceStep}{2}), \\
        \momentumX(\timeVariable, \xLabel, -\tfrac{\spaceStep}{2})
        &:=
        2\densityIncompressible \velocityX^{\textnormal{ex}}(0)
        -
        \momentumX(\timeVariable, \xLabel, \tfrac{\spaceStep}{2}), \\
        \momentumY(\timeVariable, \xLabel, -\tfrac{\spaceStep}{2})
        &:=
        2\densityIncompressible \velocityY^{\textnormal{ex}}
        -
        \momentumY(\timeVariable, \xLabel, \tfrac{\spaceStep}{2}), 
    \end{align*}
    and then
    \begin{equation}\label{eq:BCPoiseuilleLowerEquilibrium}
        \vectorial{\distributionFunction}_{\yPlusVelocity}^{\collided}(\timeVariable, \xLabel, -\tfrac{\spaceStep}{2})
        :=
        \vectorial{\distributionFunction}_{\yPlusVelocity}^{\atEquilibrium}
        (\density(\timeVariable, \xLabel, -\tfrac{\spaceStep}{2}), \momentumX(\timeVariable, \xLabel, -\tfrac{\spaceStep}{2}), \momentumY(\timeVariable, \xLabel, -\tfrac{\spaceStep}{2})).
    \end{equation}
    \item Upper boundary (at $\yLabel = L$). This is dealt with analogously to the lower boundary.
\end{itemize}

When the scheme for the momentums is not a relaxation one, \eqref{eq:BCPoiseuilleLowerEquilibrium} is not enough to preserve second-order accuracy.
This is not the case for the left and right boundary, as the exact solution of the Poiseuille flow is constant in the velocity field and linear in the pressure field along the normal vector to these interfaces.
The idea behind the correction is that we assume that the distribution functions are at equilibrium up to $\BigO{\spaceStep}$ terms.
Hence, we can use a first-order extrapolation of the non-equilibrium distribution function $\vectorial{\distributionFunction}^{\textnormal{neq}} = \vectorial{\distributionFunction}-\vectorial{\distributionFunction}^{\atEquilibrium}$ to devise the correction.
Thus, \eqref{eq:BCPoiseuilleLowerEquilibrium} becomes
\begin{multline}\label{eq:BCPoiseuilleLowerEquilibriumPlusCorrection}
        \vectorial{\distributionFunction}_{\yPlusVelocity}^{\collided}(\timeVariable, \xLabel, -\tfrac{\spaceStep}{2})
        :=
        \vectorial{\distributionFunction}_{\yPlusVelocity}^{\atEquilibrium}
        (\density(\timeVariable, \xLabel, -\tfrac{\spaceStep}{2}), \momentumX(\timeVariable, \xLabel, -\tfrac{\spaceStep}{2}), \momentumY(\timeVariable, \xLabel, -\tfrac{\spaceStep}{2}))\\
        +
        \underbrace{\vectorial{\distributionFunction}_{\yPlusVelocity}^{\collided}(\timeVariable, \xLabel, \tfrac{\spaceStep}{2})
        -
        \vectorial{\distributionFunction}_{\yPlusVelocity}^{\atEquilibrium}
        (\density(\timeVariable, \xLabel, \tfrac{\spaceStep}{2}), \momentumX(\timeVariable, \xLabel, \tfrac{\spaceStep}{2}), \momentumY(\timeVariable, \xLabel, \tfrac{\spaceStep}{2}))}_{=\vectorial{\distributionFunction}_{\yPlusVelocity}^{\textnormal{neq}}(\timeVariable, \xLabel, {\spaceStep}/{2})\quad \text{(correction)}}.
\end{multline}

\subsubsection{Convergence}

\begin{table}
    \begin{center}\caption{
        \label{tab:convergencePoiseuille}
        Empirical convergence for the Poiseuille test for different choices of parameters.
        Errors are computed at final time equal to $0.05$.}
        \begin{tabular}{|cc||c|c||c|c|}
            \hline
                & 
                & \multicolumn{2}{c||}{$\relaxationParameter_{\momentumX}=\relaxationParameter_{\momentumY} = 1$}
                & \multicolumn{2}{c|}{$\relaxationParameter_{\momentumX}=\relaxationParameter_{\momentumY} = 1.15$}\\
                \hline
                cells-per-dir. & $\spaceStep$ 
                & $L^2$ error on $\velocityX$ & order 
                & $L^2$ error on $\velocityX$ & order\\
                \hline 
                \hline
                \multicolumn{6}{|c|}{Boundary conditions \eqref{eq:BCPoiseuilleLeftEquilibrium}--\eqref{eq:BCPoiseuilleRightEquilibrium}--\eqref{eq:BCPoiseuilleLowerEquilibrium}}\\
                \hline
        \hline
        75	& 1.333E-02	& 2.938E-05	& ---  & 5.661E-04 &	--- \\
        112	& 8.929E-03	& 1.333E-05	& 1.97 & 3.812E-04 &	0.99 \\
        168	& 5.952E-03	& 5.994E-06	& 1.97 & 2.563E-04 &	0.98 \\
        253	& 3.953E-03	& 2.668E-06	& 1.98 & 1.707E-04 &	0.99 \\
        \hline 
        Average & --- & --- & 1.97 & --- & 0.99\\
        \hline 
                \hline
                \multicolumn{6}{|c|}{Boundary conditions \eqref{eq:BCPoiseuilleLeftEquilibrium}--\eqref{eq:BCPoiseuilleRightEquilibrium}--\eqref{eq:BCPoiseuilleLowerEquilibriumPlusCorrection}}\\
                \hline
        \hline
        75	& 1.333E-02	& 2.943E-05	& ---				        & 1.485E-05 & ---  \\
        112	& 8.929E-03	& 1.335E-05	& 1.97						& 6.672E-06 & 2.00 \\
        168	& 5.952E-03	& 5.997E-06	& 1.97						& 3.043E-06 & 1.94 \\
        253	& 3.953E-03	& 2.670E-06	& 1.98						& 1.368E-06 & 1.95 \\
        \hline 
        Average & --- & --- & 1.97 & --- & 1.96\\
        \hline 
    \end{tabular}
    \end{center}
\end{table}

We simulate up to a final time $0.05$.
Moreover, we consider $\relaxationParameter_{\density} = \frac{6}{5}$, $\linearEquilibrium_{\density} = \frac{3}{40}$ in every numerical simulation.
Two sets of relaxation parameters are considered for the momentums, namely $\relaxationParameter_{\momentumX}=\relaxationParameter_{\momentumY} = 1$ and $\relaxationParameter_{\momentumX}=\relaxationParameter_{\momentumY} = 1.15$.
Finally, we test with the equilibrium boundary conditions \eqref{eq:BCPoiseuilleLeftEquilibrium}--\eqref{eq:BCPoiseuilleRightEquilibrium}--\eqref{eq:BCPoiseuilleLowerEquilibrium} and with the ones featuring corrections on the lower and upper boundary \eqref{eq:BCPoiseuilleLeftEquilibrium}--\eqref{eq:BCPoiseuilleRightEquilibrium}--\eqref{eq:BCPoiseuilleLowerEquilibriumPlusCorrection}.

From the results gathered in \Cref{tab:convergencePoiseuille}, we observe two things.
The first one is that in the case where $\relaxationParameter_{\momentumX}=\relaxationParameter_{\momentumY} = 1.15$, the correction on the boundary conditions is needed to reach second-order accuracy.
The second fact is that $\relaxationParameter_{\momentumX}=\relaxationParameter_{\momentumY} = 1.15$ ensures, compared to $\relaxationParameter_{\momentumX}=\relaxationParameter_{\momentumY} = 1$, errors which are essentially divided by a factor two at any given mesh resolution.

\section{Conclusions}\label{sec:conclusions}

In this paper we have proposed a \strong{second-order accurate vectorial lattice Boltzmann method} for the approximation of the \strong{incompressible Navier-Stokes} equations, inspired by discrete kinetic formulations \cite{carfora2008discrete, bouchut2018} and corresponding relaxation schemes. 
Such approach allows relaxation far from the equilibria, with the possibility of seeking desirable numerical properties, for instance, reduced errors. 
Numerical simulations indicate that the advantages of LBMs over standard relaxation schemes are \strong{often problem-dependent}. 
Future research will establish a more general and rigorous framework to address this aspect and thoroughly describe the cases where non-equilibrium relaxation yields consistently better performances. 

This being said, we have tried to provide as many clues as possible on how to select the numerous parameters in the scheme, both theoretically through spectral analyses, and with numerical experiments.

Despite the method requires a parabolic scaling, hence many time-steps to reach final time, the simplicity of the \strong{collide-and-stream} procedure yields highly efficient and cheap computation of each iteration. 
Therefore, the method is competitive against implicit approaches, where fewer time-steps are needed but each of them carries a significant computational overhead.
Furthermore, the  proposed scheme can be easily embedded into existing efficient parallel solvers, and handle systems with additional equations.
This last point is the main advantage of vectorial schemes over those based on a scalar distribution function. 

Finally, the implementation of accurate and robust \strong{boundary conditions} that reproduce the desired physics in this framework remains essential, in particular when dealing with complex geometries.


\section*{Acknowledgement}
TT received funding from the European Union's Horizon Europe research and innovation program under the Marie Skłodowska-Curie Doctoral Network DataHyking (Grant No. 101072546). TT is member of the INdAM Research National Group of Scientific Computing (INdAM-GNCS).\\

\bibliographystyle{alpha}
\bibliography{biblio}

\end{document}